\documentclass{article}

\usepackage{amssymb}
\usepackage[latin1]{inputenc}
\usepackage{amsmath}
\usepackage{tikz-cd}
\usepackage[latin1]{inputenc}

\newtheorem{definition}{Definition}
\newtheorem{lemma}{Lemma}
\newtheorem{example}{Example}
\newtheorem{theorem}{Theorem}
\newtheorem{proposition}{Proposition}
\newtheorem{remark}{Remark}
\newtheorem{corollary}{Corollary}
\newenvironment{proof}[1][Proof]{\noindent\textbf{#1.} }{\ \rule{0.5em}{0.5em}}

\begin{document}

\title{Robustness of non-computability}
\author{Daniel S. Gra\c{c}a\\Universidade do Algarve, C. Gambelas, 8005-139 Faro, Portugal\\ and Instituto de Telecomunica\c{c}\~{o}es, Lisbon, Portugal
	\and Ning Zhong\\DMS, University of Cincinnati, Cincinnati, OH
	45221-0025, U.S.A.}
\maketitle

\begin{abstract}
		Turing computability is the standard computability paradigm which captures the computational power of digital computers. To understand whether one can create physically realistic devices which have super-Turing power, one needs to understand whether it is possible to obtain systems with super-Turing capabilities which also have other desirable properties such as robustness to perturbations.        	
	In this paper we introduce a framework for analyzing whether a
	non-computability result is robust over continuous spaces. Then we use this framework to study the degree of robustness of
	several non-computability results which involve the wave equation, differentiation, and basins of attraction.
\end{abstract}

\section{Introduction} \label{introduction}

The Turing paradigm captures the computational power of digital computers. It is currently unknown whether the Turing paradigm captures the computational power of any other physically realistic device which can be used to compute under a reasonable notion of computation. Nonetheless, several results exist which show that non-computability can occur in the setting of continuous systems. Non-computability results which pertain to the asymptotic behavior (when time goes to infinity) of a continuous system are usually coherent with the Turing paradigm since they usually are just a continuous variant of the Halting problem. For example, if a digital computer is considered as an analog machine in the sense that it uses analog voltages, etc., then its asymptotic behavior may be non-computable, but this does imply super-Turing computational power as the asymptotic behavior of Turing machines is also non-computable.

More surprising from the Turing paradigm point of view are results which allow one to obtain a non-computable quantity in a \emph{finite amount of time}, thus achieving super-Turing computational power, e.g.~as in the case of the computable wave equation from \cite{PR81}, \cite{PZ97} having a unique non-computable solution. However, these results rely on the use of systems having some properties which are not generally deemed ``realistic''. For example, such systems may have a non-unique solution \cite{Abe71}, \cite{PR79} or may not be robust to perturbations. Some cases are more subtle, as in the case of the computable wave equation admitting a unique non-computable solution (in particular this solution is not computable for $t=1$, where $t$ is the time variable) \cite{PR81}, \cite{PZ97}. Non-computability in the later case disappears if one uses an appropriate norm for the case
\cite{WZ02}.

The present paper focuses on the study of the robustness of non-computability. In particular, we introduce a framework which allows one to systematically study the degree of robustness of non-computability results over continuous spaces. We will consider some non-computability results involving the asymptotic behavior of continuous systems (e.g.~determining the basin of attraction of an ordinary differential equation) and other results where non-computability is not related to asymptotic behavior (differentiation). Using our framework, we conclude that the former (asymptotic) non-computability results are robust to an high degree, while the latter non-computability results are not robust to the same extent. Next we provide a more detailed description of the contents of this paper.

In the classical setting the notion of computability is defined over
discrete spaces such as the set $\mathbb{N}$ of non-negative
integers or the set $\{0,1\}^{*}$ of binary words. The ``digital
nature" of a function acting on $\mathbb{N}$ stems from the discrete
topology on $\mathbb{N}$. In computable analysis the notion of
computability is extended to continuous spaces (i.e., non-discrete
topological spaces) such as the set $\mathbb{R}$ of real numbers or
the set $C(\mathbb{R})$ of continuous real-valued functions on
$\mathbb{R}$ (see e.g.~\cite{Wei00}, \cite{BHW08}). The extension
depends heavily on the continuous nature of real numbers. There are
substantial differences between the continuous and the discrete
spaces which are relevant from a computability perspective. For
example, a function $f$ from $\mathbb{N}$ to $\mathbb{N}$ is always
continuous and maps every computable point to a computable point;
but the same cannot be said of a function $g$ from $\mathbb{R}$ to
$\mathbb{R}$: $g$ may map a computable real number to a
non-computable real number. If this happens, it is natural to ask
whether the non-computability spreads to a neighborhood of $x$. The
existence of non-trivial neighborhoods is a character pertaining to
continuous phenomena only.

Another theme which can occur in the context of continuous spaces is
the realism of non-computability results, as mentioned earlier.
Thus, it seems desirable to
distinguish different types of non-computability in terms of their
degree of ``physical realism."

We remark in passing that although the notion of ``physical
realism'' or ``physical relevance'' of some property of a system is
not precisely defined, it is commonly accepted by physicists that a
``physically relevant'' property should be robust under small
perturbations. This is because errors in observations and in
measurements are inherent and unavoidable to any physical systems.
We note that the idea of ``physical relevance'' implying robustness
to (appropriate) perturbations is not new and has been used in
various fields, e.g., in the dynamical system theory (see
e.g.~\cite[p.~259]{GH83}). It is clear that robustness under perturbations (i.e.,
nearby points share the same qualitative property) is also
pertaining to continuous spaces only.

In the continuous setting the non-computability of an operator is
usually the result of either (i) discontinuity,
or (ii) an explicit construction of the operator producing the
non-computable image on a computable input (which may happen asymptotically). The latter is stronger
because there are discontinuous (and hence non-computable) operators
which take every computable element to a computable element. On the
other hand, since an explicit construction usually makes use of a
single point in the domain of the operator, often the
non-computability has no ``interaction" with the topology of the
domain - it does not provide information on computability of nearby
points such as whether the operator remains non-computable for the
nearby points, for some of the nearby points, or the operator
becomes computable for the nearby points; in other words, it does
not provide information on whether the non-computability is robust
under perturbations.

In this paper, we study the robustness of non-computability over
continuous spaces. To the best of our knowledge, this issue has not
been addressed in the field of computable analysis and real
computation. Our main contribution is to introduce a framework for analyzing whether a
non-computability result is robust with respect to the relevant
topologies. Then we will use this framework to study the degree of
robustness of some non-computability results concerning the wave equation, differentiation, and basins of attraction. Recall
that the basin of attraction of an attractor is the set of points
which eventually converge to this attractor. We observe that there
is a close resemblance between a basin of attraction and the Halting
problem in the discrete setting.

The main results of the paper are summarized below using the notions
of robustness in non-computability -- robust non-computability
and semi-robust non-computability -- introduced in section 3. The results
are stated informally. The precise statements are given in sections
4, 5 and 6.

\begin{itemize}
\item[(1)] The differentiation operator $D: C^1(\mathbb{R}) \to
C(\mathbb{R})$, $f\mapsto f^{\prime}$, is robustly non-computable
at some functions in $C^1(\mathbb{R})$. That is, there exist a
function $f\in C^1(\mathbb{R})$ and a $C^1$-neighborhood of $f$ such
that the derivative $g^{\prime}$ is non-computable for every $g$ in
this neighborhood. But the claim breaks down if $\mathbb{R}$ is
replaced by a compact interval $[a, b]$, where $a$ and $b$ are
computable real numbers: the differentiation $D:C^1([a, b]) \to
C([a, b])$ cannot be robustly non-computable at any function in
$C^1([a, b])$; on the other hand, $D$ can be semi-robustly non-computable
at some functions in $C^1([a, b])$. That is, there exists a function
$f\in C^1([a, b])$ such that for any $C^1$-neighborhood $U$ of $f$,
there exists some computable function $g\in U$ whose derivative
$g^{\prime}$ is non-computable. These results are presented in
section 4.

\item[(2)] The map $(h, s_h)\mapsto W_{s_h}$ is robustly non-computable
at some computable and analytic function $f: \mathbb{R}^3 \to
\mathbb{R}^3$, where $h\in C^1(\mathbb{R}^3)$, $s_h$ is a sink (i.e.~an attracting equilibrium point) of
$h$ (computable from $h$), and $W_{s_h}$ is the basin of attraction
of $s_h$. This result is presented in section 5.

\item[(3)] Let $K$ be a closed disk centered at the origin of
$\mathbb{R}^2$ with a rational radius. Then the map $(h, s_h)
\mapsto W_{s_h}$ is not robustly non-computable at any $h\in
\mathcal{V}(K)$, where $\mathcal{V}(K)$ is the set of all $C^1$
vector fields $K\to \mathbb{R}^2$ pointing inwards along the
boundary of $K$. On the other hand, the map $(h, s_h) \mapsto
W_{s_h}$ can be semi-robustly non-computable at some computable $f\in
\mathcal{V}(K)$. These results are discussed in section 6.
\end{itemize}

\section{Preliminaries}

\label{Sec:prelims}

\subsection{Computable analysis}

\label{Subsec:comptanalysis}

There are several different yet equivalent approaches to computable
analysis going back to Grzegorczyk and Lacombe in the 1950s. In this
paper, we use the representation version.

Roughly speaking, an object is computable if it can be approximated
by computer-generated approximations with an arbitrarily high
precision. Formalizing this idea to carry out computations on
infinite objects such as real numbers, we encode those objects as
infinite sequences of rational numbers (or equivalently, sequences
of any finite or countable set $\Sigma$ of symbols), using
representations (see \cite{Wei00} for a complete development). A
represented space is a pair $(X; \delta)$ where $X$ is a set,
$\delta$ is a coding system (or naming system) on $X$ with codes
from $\Sigma$ having the property that
$\mbox{dom}(\delta)\subseteq\Sigma^{\mathbb{N}}$ and $\delta:
\Sigma^{\mathbb{N}}\to X$ is an onto map. Every
$q\in\mbox{dom}(\delta)$ satisfying $\delta(q)=x$ is called a
$\delta$-name of $x$ (or a name of $x$ when $\delta$ is clear from
context). Naturally, an element $x\in X$ is computable if it has a
computable name in $\Sigma^{\mathbb{N}}$. The notion of
computability on $\Sigma ^{\mathbb{N}}$ is well established, and
$\delta$ lifts computations on $X$ to computations on
$\Sigma^{\mathbb{N}}$. The representation $\delta$ also induces a
topology $\tau_{\delta}$ on $X$, where $\tau_{\delta} = \{
U\subseteq X: \, \delta^{-1}(U)\text{ is open in }\operatorname{dom}
(\delta)\}$ is called the final topology of $\delta$ on $X$.

The notion of computable maps between represented spaces now arises
naturally. A map $\Phi: (X;\delta_{X})\to(Y;\delta_{Y})$ between two
represented spaces is computable if there is a computable map $\phi:
\Sigma^{\mathbb{N}}\to\Sigma^{\mathbb{N}}$ such that
$\Phi\circ\delta_{X}=\delta_{Y}\circ\phi$  as depicted below (see
e.g. \cite{BHW08}).
\begin{center}
    \begin{tikzcd}
        \Sigma^{\mathbb{N}} \arrow[r, "\phi"] \arrow[d, "\delta_{X}"]
        & \Sigma^{\mathbb{N}} \arrow[d, "\delta_{Y}"] \\
        X \arrow[r, "\Phi"]
        & Y
    \end{tikzcd}
\end{center}
Informally speaking, this means that there is a computer program
$\phi$ that outputs a name of $\Phi(x)$ when given a name of $x$ as
input. Since $\phi$ is computable, it transforms every computable
element in $\Sigma^{\mathbb{N}}$ to a computable element in
$\Sigma^{\mathbb{N}}$; hence the map $\Phi$ takes every computable
point in $X$ to a computable point in $Y$. If $\Phi$ takes every
computable point in $X$ to a computable point in $Y$, then it is
called computably invariant (terminology introduced in \cite{Bra99}). A computably invariant map is not necessarily
computable \cite{Bra99}. Another fact about computable maps is that
computable maps are continuous with respect to the corresponding
final topologies induced by $\delta_X$ and $\delta_Y$.

In the later sections, we will consider non-computability of
operators having domain contained in the set $C(W)$ of continuous
(or $C^1(W)$ of continuously differentiable) real-valued functions
defined on some open subset $W$ of $\mathbb{R}^d$. Hence, it is
perhaps in order to recall the standard topologies on these sets.
For $x\in \mathbb{R}^d$, $\| x\|= \max_{1\leq i\leq d}|x_i|$.  For
$f\in C(W)$ (resp. $f\in C^1(W)$), the $C$-norm of $f$ is defined to
be $\|f\|=\sup_{x\in W}|f(x)|$ (resp. the $C^1$-norm of $f$ is
defined to be $\|f\|_1 = \sup_{x\in W}\{ |f(x)|, \| Df(x)\|\}$). The
$C$-norm (resp. the $C^1$-norm) induces a metric and thus a
(uniform) topology on $C(W)$ (resp. $C^1(W)$) as follows: $d_{C}(f,
g) = \min\{ \| f-g\|, 1\}$ for $f, g\in C(W)$; $d_{C^1}(f, g)=
\min\{ \| f - g \|_1, 1\}$. In analysis and dynamical systems,
$C$-norm and $C^1$-norm are often used to measure the distance
between any two functions in $C(W)$ or in $C^1(W)$. If $K$ is a
compact subset of $W$, then the $C$-norm and the $C^1$-norm are
defined by $\| f\| = \max_{x\in K}\|f(x)\|$ for $f\in C(K)$ and  $\|
f\|_{1} = \max_{x\in K}\|f(x)\| + \max_{x\in K}\|Df(x)\|$ for $f\in
C^1(K)$.

In this paper, we use the following particular representations for
points in $\mathbb{R}^{d}$; for open subsets of $\mathbb{R}^{d}$;
and for functions in $C(\mathbb{R}^d)$ and in $C^1(\mathbb{R}^d)$:
\begin{itemize}
\item[(1)] For a point $x\in\mathbb{R}^{d}$, a name of $x$ is a sequence $\{
r_{k}\}$ of points with rational coordinates satisfying
$\|x-r_{k}\|<2^{-k}$. Thus $x$ is computable if there is a Turing
machine (or a computer program or an algorithm) that outputs a
rational $d$-tuple $r_{k}$ on input $k$ such that
$\|r_{k}-x\|<2^{-k}$; for a sequence $\{ x_{j}\}$,
$x_{j}\in\mathbb{R}^{d}$, a name of $\{x_{j}\}$ is a double sequence
$\{ r_{j,k}\}$ of points with rational coordinates satisfying
$\|x_{j}-r_{j, k}\|<2^{-k}$.

\item[(2)] For an open subset $U$ of $\mathbb{R}^{d}$, a name of $U$ consists
of a pair of an inner-name and an outer-name; an inner-name is a
sequence of open balls $B(a_{k}, r_{k})=\{ x\in\mathbb{R}^{d} \, :
\, d(a_{k}, x)<r_{k}\}$, $a_{k}\in\mathbb{Q}^{d}$ and
$r_{k}\in\mathbb{Q}$, exhausting $U$, i.e.,
$U=\bigcup_{k=1}^{\infty}B(a_{k}, r_{k})$; an outer name is a
sequence dense in $\mathbb{R}^d\setminus U$. $U$ is said to be
r.e.~open if the sequences $\{ a_{k}\}$ an $\{ r_{k}\}$ are
computable; $U$ is said to be co-r.e.~open if the sequence (dense in
$\mathbb{R}^d\setminus U$) is computable; and $U$ is said to be
computable if it is r.e.~and co-r.e..

\item[(3)] For every $f\in C(\mathbb{R}^d)$ (resp. $f\in C^1(\mathbb{R}^d)$),
a name of $f$ is a double sequence $\{ P_{k, l}\}$ of polynomials
with rational coefficients satisfying $\| f - P_{k, l}\| <2^{-k}$
(resp. $\| f - P_{k, l}\|_{1} <2^{-k}$) on the closed ball
$\overline{B}(0, 2^l)\subset \mathbb{R}^d$. A name for a
vector-valued function consists of the names of its component
(real-valued) functions. Thus, $f$ is computable (resp.
$C^1$-computable) if there is a Turing machine that outputs a
vector-valued function with polynomial components $P_{k, l}$ (more
precisely coefficients of $P_{k, l}$) on input $k, l$ satisfying $\|
P_{k, l} - f \| <2^{-k}$ (resp. $\| P_{k, l} - f \|_{1} < 2^{-k}$)
on $\overline{B}(0, 2^{l})$.
\end{itemize}

By the definition, a planar computable bounded open set can be
visualized on a computer screen with an arbitrarily high
magnification. For a closed subset $D$ of $\mathbb{R}^d$, $D$ is
said to be computable if its complement $\mathbb{R}^d\setminus D$ is
a computable open subset of $\mathbb{R}^d$; or equivalently, the
distance function $d_{D}: \mathbb{R}^d \to \mathbb{R}$, $d_{D}(x) =
\inf _{y\in D} \| y - x\|$, is a computable function. \\

\subsection{Dynamical systems}

\label{Subsec:DynamicalSystem}

We recall that there are two broad classes of dynamical systems:
discrete-time and continuous-time (for a general definition of
dynamical systems, encompassing both cases, see \cite{HS74}). A
discrete-time dynamical system is defined by the iteration of a map
$f:\mathbb{R}^{d}\rightarrow\mathbb{R}^{d}$, while a continuous-time
system is defined by an ordinary differential equation (ODE)
$x^{\prime}=f(x)$. Common to both classes of systems is the notion
of trajectory: in the discrete-time case, a \emph{trajectory}
starting at the
point $x_{0}$ is the sequence
\[
x_{0},f(x_{0}),f(f(x_{0})),\ldots,f^{[k]}(x_0),\ldots
\]
where $f^{[k]}$ denotes the $k$th iterate of $f$, while in the
continuous time case it is the solution, a function $\phi (f,
x_0)(\cdot)$ of time $t$,  to the following
initial-value problem
\[
\left\{
\begin{array}
[c]{l}
x^{\prime}=f(x)\\
x(0)=x_{0}
\end{array}
\right.
\]

A trajectory may converge to some \emph{attractor}. Attractors are
invariant sets in the sense that if a trajectory reaches an
attractor, it stays there.
Given an attractor $A$, its \emph{basin of attraction} is the set
\[
\{x\in\mathbb{R}^{d}|\text{ the trajectory starting at }x\text{
converges to }A \text{ as } t\rightarrow\infty \}
\]
Attractors come in different varieties: they can be points, periodic
orbits, strange attractors, etc. The equilibrium points (also called
\textit{fixed points} in the discrete case) are the simplest
attractors. If there is a neighborhood around an equilibrium point
$s$ which is contained in the basin of attraction of $s$ (i.e.\
every trajectory starting in this neighborhood will converge to
$s$), then $s$ is called a sink.

\section{Robustness of non-computability - definitions}

\label{Sec:RobustnessNoncomputability}

We now introduce several notions of robust non-computability. In
analysis and dynamical systems, perturbations (and thus robustness
under perturbations) are typically described and measured in terms
of certain norm, metric, or topology on the space under
consideration. We follow the classical approach in the following
definitions.

\begin{definition}\label{Def:robust} Let $X$ and $Y$ be represented spaces; let
$\Phi: X\to Y$ be a map; and let $\tau_{X}$ be a topology on $X$.
Assume that $x_c$ is a computable element of $X$ and
$\Phi(x_c)=y_{nc}$ is a non-computable element of $Y$.
\begin{itemize}
\item[(1)] $\Phi$ is called $\tau_{X}$-robustly non-computable at $x_c$ if
there is a $\tau_{X}$-neighborhood $U$ of $x_c$ in $X$, i.e., $U\in
\tau_{X}$ and $U$ contains $x_c$, such that for any $x\in U$,
$\Phi(x)$ is non-computable.
\item[(2)] $\Phi$ is called $\tau_{X}$-semi-robustly  non-computable at $x_c$ if for
any $\tau_{X}$-neighborhood $U$ of $x_c$, there exists a computable
element $x_U\in U$ such that $x_U\neq x_c$ and $\Phi(x_U )$ is
non-computable.
\item[(3)] Let $A$ be a subset of $X$. Then $\Phi$ is called
$A$-removably non-computable at $x_c$ with respect to $\tau_{X}$ if
there is a $\tau_{X}$-neighborhood $U$ of $x_c$ such that $U\cap A$
is dense in $U$ and $\Phi$ is computable on $(U\cap A)\setminus \{
x_c\}$. If $(U\cap A)\setminus \{ x_c\} = U\setminus \{ x_c\}$, then
$\Phi$ is called removably non-computable at $x_c$.
\end{itemize}
\end{definition}

Recall that every represented space has a set of computable points
corresponding to computable points in $\Sigma^{\mathbb{N}}$ via the
representation.

By Definition \ref{Def:robust}, if the topology $\tau_X$ is not
discrete and $\Phi$ is $\tau_{X}$-robustly non-computable at $x_c$,
then the non-computability of $\Phi$ at $x_c$ is a ``continuous"
property near $x_c$ with respect to $\tau_{X}$; in this case, the
non-computability is intrinsic and persistent in a neighborhood of
$x_c$. If $\tau_X$ is induced by a norm or a metric, then the
non-computability is persistent under small perturbations.

It is readily seen that if $\Phi$ is robustly non-computable at
$x_c$, then it is semi-robustly non-computable at $x_c$ provided
that every neighborhood of $x_c$ contains a computable point
different from $x_c$; the converse is false; and $\Phi$ cannot be
removably non-computable at $x_c$ if it is semi-robustly
non-computable at $x_c$.

A notational matter. In the rest of the paper, we will use the
standard topology on $\mathbb{R}^d$. For the function spaces $C(W)$
and $C^1(W)$, we will use the uniform topologies induced by $\|
\cdot\|$ and $\| \cdot \|_1$, respectively (see section
\ref{Subsec:comptanalysis} for the definitions of these norms).
Therefore, in the rest of the paper, we will omit the prefix
$\tau_X$ in the definition of robust, of semi-robust, and of
removable non-computability for operators defined on those spaces.

We now present an example of each type of non-computable function
according to Definition \ref{Def:robust}.

\begin{example}
\begin{itemize}
\item[(a)] The constant function $g: [-1, 1] \to \mathbb{R}$, $x\mapsto
\alpha$, is robustly non-computable, where $\alpha$ is a
non-computable real.
\item[(b)] The function $h: \mathbb{R}\to
\mathbb{R}$ is semi-robustly non-computable, where
\[ h(x) =  \left\{ \begin{array}{ll}
                1 & \mbox{if $x$ is a rational number} \\
                \alpha & \mbox{if $x$ is an irrational number}
                \end{array} \right. \]
The non-computability in $h$ is not removable at any real number,
and $h$ is not robustly non-computable.
\item[(c)] The following
is an example of a function that is computable everywhere except at
$x=0$; in other words, the function is removably non-computable at
$x=0$. Let $a:\mathbb{N\rightarrow N}$ be a one
    to one recursive function generating a recursively enumerable noncomputable
    set $A$. Then it is well known (see e.g. \cite[pp. 16--17]{BR89}) that the
    real number
    \[
    s=\sum_{m=0}^{\infty}2^{-a(m)}
    \]
    exists and is noncomputable. Let us consider the piecewise affine map
    $\phi:\mathbb{R\rightarrow R}$ defined as follows:

    \begin{enumerate}
        \item $\phi(x)=2^{-a(0)}+2^{-a(1)}$ for $x\geq1$;

        \item The graph of $\phi$ on $[1/(n+1),1/n]$, for any $n\in\mathbb{N}
        \backslash\{0\}$, is the line segment between the points $(1/(n+1),\sum
        _{m=0}^{n+1}2^{-a(m)})$ and $(1/n,\sum_{m=0}^{n}2^{-a(m)})$;

        \item $\phi(0)=\sum_{m=0}^{\infty}2^{-a(m)}$;

        \item $\phi(x)=\phi(-x)$ for $x<0$.
    \end{enumerate}

    It is not difficult to see that $\phi$ is continuous. It follows from its
    definition that $\phi$ is uniformly computable in $\mathbb{R}\backslash\{0\}$,
    which implies that for any computable $x\in\mathbb{R}\backslash\{0\}$ one has
    that $\phi(x)$ is computable. Furthermore $\phi(0)$ is not computable, which
    implies that $\phi$ is removably non-computable
    at $0$.
\end{itemize}
\end{example}

\section{Non-computability of differentiation}
\label{sec:Non-computabilityDifferentiation}

It is well-known that the differentiation operator is
non-computable. In this section, we discuss the robustness of
non-computability in differentiation.

We begin with the compact domain $[a, b]$, where $a$ and $b$ are
computable reals. Since the set of polynomials with rational
coefficients is dense in $C^1([a, b])$ and the derivative of a
polynomial with rational coefficients is again a polynomial with
rational coefficients (thus computable), it follows that the
differentiation $D: C^1([a, b]) \to C([a, b])$ cannot be robustly
non-computable at any function in $C^1([a, b])$. On the other hand,
it is possible for $D$ to be semi-robustly non-computable at some
functions in $C^1([a, b])$. To see that $D$ is semi-robustly
non-computable at some functions in $C^1([a, b])$, we note that in
\cite[Theorem 1 on page 51]{PR89} it is shown that there is a
computable function $f\in C^1([0, 1])$ such that $f^{\prime}$ is not
computable. For any neighborhood $U$ of $f$ in $C^1[0, 1]$, there is
rational number $\varepsilon >0$ such that the closed ball
$\bar{B}(f,\varepsilon)=\{g\in C^1([a,
b]):\|f-g\|_1\leq\varepsilon\}\subseteq U$. It is clear that
$f+\varepsilon \in \bar{B}(f,\varepsilon)$, $f+\varepsilon$ is
computable, and $(f+\varepsilon)^{\prime} = f^{\prime}$ is
non-computable. Consequently, $D$ is semi-robustly non-computable at
$f$.

The next theorem shows that differentiation can be
robustly non-computable if the compact domain $[a, b]$ is replaced
by $\mathbb{R}$.

\begin{theorem} \label{th:DifferentiationEverywhereNoncomputable}
There exists a function $f\in C^{\infty}(\mathbb{R})$ such that $f$
is computable and $D: C^1(\mathbb{R}) \to C(\mathbb{R})$, $f\mapsto
f^{\prime}$, is robustly non-computable at $f$.
\end{theorem}

\begin{proof} Again we make use of a function $f$ constructed by Pour-El and
Richards \cite[Remark A on page 58]{PR89}. We start by recalling the
construction of $f$. Let $\phi$ be a canonical $C^{\infty}$ bump
function:
\[ \phi(x) = \left\{ \begin{array}{ll}
e^{-x^2/(1-x^2)} & \mbox{for $|x|<1$} \\
0 & \mbox{for $|x|\geq 1$} \end{array} \right.
\]
Then $\phi \in C^{\infty}(\mathbb{R})$, $\phi$ has support on $[-1,
1]$, and $\phi(0)=1$. Let $\psi(x) =
\phi[x - (1/2)]$, so that $\psi^{\prime}(0) = \phi^{\prime}(-
1/2)>0$. Set $\psi_k(x) = (1/k) \cdot \psi (k^2x)$. As previously,
let $a: \mathbb{N}\to \mathbb{N}$ is a recursive function generating
a recursively enumerable nonrecursive set $A$ in a one-to-one
manner.  Define
\[ f(x) = \sum_{k=2}^{\infty} \psi_k(x - a(k)) \]
Then $f \in C^{\infty}$ and is computable on $\mathbb{R}$. Moreover,
\[ f^{\prime}(n) = \left\{ \begin{array}{ll}
                          k\phi^{\prime}(- 1/2) & \mbox{if $n=a(k)$}
                          \\
                          0 & \mbox{otherwise} \end{array} \right.
                          \]
Pour-El and Richards showed that $f^{\prime}$ cannot be a computable
function by way of a contradiction.  Suppose $f^{\prime}$ were
computable on $\mathbb{R}$. Then the sequence $\{ f^{\prime}(n)\}$
is computable, which would imply that the sequence $\{ \mu_n\}$
\[ \mu_{n} = \frac{f^{\prime}(n)}{\phi^{\prime}(- 1/2)} \]
is also computable. Since $n\in A$ if and only if $n=a(0)$ or
$n=a(1)$ or $\mu_n\geq 2$, it follows that $A$ would be recursive if
the sequence $\{ \mu_n\}$ is computable. We have arrived at a
contradiction.

To show that $D$ is robustly non-computable at $f$, we need to
find a $C^1$-neighborhood $U$ of $f$ such that $g^{\prime}$ is
non-computable for every $g\in U$. We start by choosing a positive
rational number $\alpha$ such that $\alpha \cdot \phi^{\prime}(-
1/2) > 3$. Set $U = \{ g\in C^1(\mathbb{R}): \|f-g\|_1 <
\frac{1}{\alpha} \}$. For any $g\in U$, let
\[ \mu^{g}_{n} = \frac{g^{\prime}(n)}{\phi^{\prime}(- 1/2)} \]
Consequently, if $n=a(k)$ for some $k\geq 2$, then
\[ \mu^{g}_{n} = \frac{g^{\prime}(n)}{\phi^{\prime}(- 1/2)} \geq
\frac{f^{\prime}(n) - (1/\alpha)}{\phi^{\prime}(- 1/2)} =
\frac{k\phi^{\prime}(- 1/2) - (1/\alpha)}{\phi^{\prime}(- 1/2)} \geq
2 - \frac{1}{3}=\frac{5}{3} \] On the other hand, if $n\neq a(k)$
for all $k$, then
\[ \mu^{g}_{n} = \frac{g^{\prime}(n)}{\phi^{\prime}(- 1/2)} \leq
\frac{f^{\prime}(n)+ (1/\alpha)}{\phi^{\prime}(- 1/2)} =
\frac{1}{\alpha \phi^{\prime}(- 1/2)} < \frac{1}{3} \] In other
words, $n\in A$ if and only if one of the following holds: $n=a(0)$
or $n=a(1)$ or $\mu^{g}_{n}\geq \frac{5}{3}$. Since $A$ is
nonrecursive, the sequence $\{ \mu^{g}_{n}\}$ cannot be computable,
which implies that the sequence $\{ g^{\prime}(n)\}$ cannot be
computable. Hence, $g^{\prime}$ is not a computable function.
\end{proof}

\section{Robust non-computability}

\label{sec:EverywhereNoncomputable}

In this section, we discuss  the map that assigns to $(f, s)$ the
basin of attraction of $s$, where $f$ is a 3-dimensional $C^1$
discrete-time evolution function and $s$ is a sink of $f$. We show
that the map can be robustly non-computable at an analytic and
computable function $f$.

In \cite{GZ15} an analytic and computable function $f: \mathbb{R}^3 \to
\mathbb{R}^3$ is explicitly constructed with the following
properties:
\begin{itemize}
\item[(a)] the restriction, $f_{M}: \mathbb{N}^3 \to \mathbb{N}^3$, of $f$ on
$\mathbb{N}^3$ is the transition function of a universal Turing
machine $M$, where each configuration of $M$ is coded as an element
of $\mathbb{N}^3$ (see \cite{GZ15} for an exact description of the
coding). Without loss of generality, $M$ can be assumed to have just
one halting configuration; e.g. just before ending, set all
variables to $0$ and go to some special line with a command break;
thus the final configuration is unique;
\item[(b)] the halting configuration $s$ of $M$ is a computable
sink of the discrete-time evolution of $f$ (see Section
\ref{Subsec:DynamicalSystem} for the definition of sink).
\item[(c)] the basin of attraction of $s$ is non-computable (see Section
\ref{Subsec:DynamicalSystem} for the definition of basin of
attraction).
\item[(d)] there exists a constant $\lambda \in (0, 1)$ such that if $x_0$
is a configuration of $M$, then for any $x\in \mathbb{R}^3$,
\begin{equation} \label{f-contraction}
\| x - x_0\| \leq 1/4 \quad \Longrightarrow \quad \| f(x) -
f(x_0)\|\leq \lambda \| x - x_0\|
\end{equation}
\end{itemize}

In the remaining of this section, the symbols $f$ and $s$ are
reserved for this particular function and its particular sink - the
halting configuration of the universal Turing machine $M$ whose
transition function is $f_M$.

We show in the following that there is a $C^1$-neighborhood
$\mathcal{N}$ of $f$ -- computable from $f$ and $Df(s)$ -- such that
for every $g\in \mathcal{N}$, $g$ has a sink $s_{g}$ -- computable
from $g$ -- and the basin of attraction of $s_{g}$ is
non-computable. We begin with two propositions. (Note that the
particular function $f$ is exactly the function constructed in
\cite{GCB08} mentioned in the following proposition.)

\begin{proposition} \label{Turing-robustness} (\cite[p. 333]{GCB08}) Let $0< \delta < \epsilon < \frac{1}{2}$.
The extension from $f_{M}$ to $f$ is robust to perturbations in the
following sense: for all $g: \mathbb{R}^3 \to \mathbb{R}^3$ such
that $\| f - g\|\leq \delta$, for all $j\in \mathbb{N}$, and for all
$\bar{x}_0\in \mathbb{R}^3$ satisfying $\| \bar{x}_0 - x_0\| \leq
\epsilon$, where $x_0\in \mathbb{N}^3$ represents an initial
configuration,
\begin{equation} \label{perturbation}
\| g^{[j]}(\bar{x}_0) - f_{M}^{[j]}(x_0)\| \leq \epsilon
\end{equation}
\end{proposition}

\begin{proposition}  \label{hyperbolic-robustness} 
There exist a neighborhood $U\subset \mathbb{R}^3$ of $s$ and a
$C^1$-neighborhood $\mathcal{N}$ of $f$ such that for any $g\in
\mathcal{N}$ there is a unique equilibrium $s_{g}\in U$ of $g$.
Moreover, for any $\epsilon
>0$ one can choose $\mathcal{N}$ so that $\| s_{g}-s\| < \epsilon$.
\end{proposition}

The proposition \ref{hyperbolic-robustness} is a discrete-time
version of Theorem 1 in \cite[p. 305]{HS74}. The proof of this
theorem has nothing to do with differential equations; rather, it
depends on the invertibility of $Df(s)$ (the invertibility of
$Df(s)$ follows from the fact that $s$ is a sink).

\begin{remark} \label{k_0} Applying a computable version of the inverse function theorem,
it can be further shown that from $f$ and $Df(s)$ one can compute an
integer $k_0$ and then two computable functions $\eta: \{
k\in\mathbb{N}: \, k\geq k_0\} \to \mathbb{N}$ (assuming $\eta(k)>k$
without loss of generality) and $F: \mathcal{N} \to \mathbb{R}^3$
such that for every $g\in \mathcal{N}$, $F(g)=s_{g}$, $s_{g}$ is a
unique equilibrium point of $g$ in the closure of $B(s, 1/k_0)$,
$s_{g}$ is a sink, and $\| s - s_g\| < 1/k$,  where $\mathcal{N}$ is
the $\frac{1}{\eta(k)}$-neighborhood (in $C^1$-norm) of $f$.
\end{remark}

\begin{theorem} \label{noncomputability-robustness}
There is a $C^1$-neighborhood $\mathcal{N}$ of $f$ (computable from
$f$ and $Df(s)$) such that for any $g\in \mathcal{N}$, $g$ has a
sink $s_g$ (computable from $g$) and the basin of attraction $W_g$
of $s_g$ is non-computable.
\end{theorem}

\begin{remark} The theorem may also be stated as follows: The map $C^1(\mathbb{R}^3) \to \mathcal{O}_3$, $g\mapsto W_g$, is
robustly non-computable at $f$, where $\mathcal{O}_d$ is the set
of all open subsets of $\mathbb{R}^d$ equipped with the topology
generated by the open rational balls (balls with rational centers
and rational radius) as a subbase.
\end{remark}

\begin{proof} Let
$0<\delta<\epsilon < \frac{1}{2}$ be the parameters given in
Proposition \ref{Turing-robustness} satisfying $0<\epsilon < 1/4$ and let $k_0$, $\eta$ be given as in Remark \ref{k_0}.
Pick a $k\geq k_0$ such that $0 < \frac{1}{k} <\epsilon/2$. Let
$\theta$ be a (rational) constant satisfying $0< \theta < \min\{
\delta, \frac{1-\lambda}{2}, \frac{1}{\eta(k)}\}$, where $\lambda
\in (0, 1)$ is the constant in (\ref{f-contraction}). Denote $\theta
+ \lambda$ as $\theta_{\lambda}$. Then $0< \theta_{\lambda} \leq
\frac{1-\lambda}{2}+\lambda = \frac{1+\lambda}{2} <1$. Let
$\mathcal{N}$ be the $\theta$-neighborhood of $f$ in $C^1$-norm.
Then for any $g\in \mathcal{N}$, any configuration $x_0\in
\mathbb{N}^3$ of $M$, and any $x\in B(x_0, 1/4)$, we have the
following estimate:
\begin{eqnarray} \label{g-contraction}
& & \| g(x) - g(x_0) \| \nonumber \\
& \leq & \| (g-f)(x) - (g-f)(x_0) \| + \| f(x) - f(x_0) \| \nonumber \\
& \leq & \| D(g-f)\|\, \|x-x_0\| + \lambda \| x - x_0 \| \nonumber \\
& < & (\theta + \lambda) \, \| x - x_0\|
\end{eqnarray}
Since $0< \theta + \lambda = \theta _{\lambda} <1$, it follows that
$g$ is a contraction in $B(x_0, 1/4)$ for every configuration $x_0$
of $M$. \\

We now show that for any $g \in \mathcal{N}$ and any configuration
$x_0$ of $M$,  $M$ halts on $x_0$ if and only if $x_0\in W_g$, where
$W_g$ denotes the basin of attraction of $s_g$. First we assume that
$x_0\in W_g$. Then $g^{[j]}(x_0)\to s_{g}$ as $j\to \infty$. Hence,
there exists $n\in \mathbb{N}$ such that $\| g^{[n]}(x_0) - s_{g}\|
< \frac{\epsilon}{5}$, which in turn implies that
\begin{eqnarray*}
& & \| f_{M}^{[n]}(x_0) - s \| \\
& \leq & \| f_{M}^{[n]}(x_0) - g^{[n]}(x_0) \| + \| g^{[n]}(x_0) -
s_{g} \| + \| s_{g} - s \| \\
& \leq & \epsilon + \frac{\epsilon}{5} + \frac{1}{k} < 2\epsilon
\end{eqnarray*}
Since $f_{M}^{[n]}(x_0), s \in \mathbb{N}^3$ and $2\epsilon < 1/2$
by assumption that $\epsilon < \frac{1}{4}$, it follows that
$f_{M}^{[n]}(x_0)=s$. Hence, $M$ halts on $x_0$. Next we assume that
$M$ halts on $x_0$. This assumption implies that there exists $n\in
\mathbb{N}$ such that $f_{M}^{[j]}(x_0)=s$ for all $j\geq n$. Then
for all $j\geq n$, it follows from Proposition
\ref{Turing-robustness} that
\begin{eqnarray*}
& & \| g^{[j]}(x_0) - s \| \\
& \leq & \| g^{[j]}(x_0)-f_{M}^{[j]}(x_0) \| + \| f_{M}^{[j]}(x_0) -
s \| \\
& = & \| g^{[j]}(x_0)-f_{M}^{[j]}(x_0) \| \leq \epsilon
\end{eqnarray*}
The inequality implies that $\{ g^{[j]}(x_0)\}_{j\geq n}\subset
\overline{B(s, \epsilon)}$. From the assumption that $s_g$ is an
equilibrium point of $g$, $\| s - s_g\| < \frac{1}{k}$ for every $g$
in the $\frac{1}{\eta(k)}$-neighborhood of $f$ (in $C^1$-norm),
$\theta\leq\frac{1}{\eta(k)}$ and $\frac{1}{k}<\epsilon$, it follows
that $g(s_g)=s_g$ and $s_g\in \overline{B(s, \epsilon)} \subset
\overline{B(s, 1/4)}$. Since $s$ is a configuration of $M$ -- the
halting configuration of $M$,  it follows from (\ref{g-contraction})
that $g$ is a contraction on $\overline{B(s, 1/4)}$. Thus, $\|
g^{[n+j]}(x_0) - s_g \| = \| g^{[n+j]}(x_0) - g^{[n+j]}(s_g) \| \leq
(\theta_{\lambda})^j \| g^{[n]}(x_0) - s_g \| \to 0$ as $j\to
\infty$. Consequently, $g^{[j]}(x_0) \to s_g$ as $j\to \infty$, This
implies that $x_0\in W_g$.

To prove that $W_g$ is non-computable, the following stronger
inclusion is needed: if $M$ halts on $x_0$, then $B(x_0,
\epsilon)\subset W_g$. Consider any $x\in B(x_0, \epsilon)$. Since
$x_0\in W_g$ and $g$ is a contraction on $B(x_0, \epsilon)$, it
follows that
\[ \| g^{[j]}(x) - g^{[j]}(x_0) \| \leq (\theta_{\lambda})^j \| x -
x_0\| \to 0 \quad \mbox{as $j\to \infty$} \] Since $x_0\in W_g$,
$g^{[j]}(x_0)\to s_g$ as $j\to \infty$. Hence, $g^{[j]}(x)\to s_g$
as $j\to \infty$. This implies that  $x\in W_g$.

It remains to show that $W_{g}$ is non-computable. Suppose otherwise
that $W_{g}$ were computable. Then the distance function
$d_{\mathbb{R}^3\setminus W_{g}}$ is computable. For every initial
configuration $x_0\in \mathbb{N}^3$, compute
$d_{\mathbb{R}^3\setminus W_{g}}(x_0)$ and halt the computation if
either $d_{\mathbb{R}^3\setminus W_{g}}(x_0)>\frac{\epsilon}{5}$ or
$d_{\mathbb{R}^3\setminus W_{g}}(x_0)<\frac{\epsilon}{4}$. Clearly
the computation will halt for every initial configuration $x_0$. Now
if $d_{\mathbb{R}^3\setminus W_{g}}(x_0)>\frac{\epsilon}{5}>0$ then
$x_0\in W_{g}$ or, equivalently, $M$ halts on $x_0$; otherwise, if
$d_{\mathbb{R}^3\setminus W_{g}}(x_0)<\frac{\epsilon}{4}$, then
$x_0\not\in W_{g}$ or, equivalently, $M$ doesn't halt on $x_0$. The
exclusion that $x_0\not\in W_{g}$ is derived from the fact that if
$x_0\in W_{g}$, then  $B(x_0, \epsilon)\subset W_{g}$; in other
words, $d_{\mathbb{R}^3\setminus W_{g}}(x_0)\geq \epsilon
> \frac{\epsilon}{4}$. Hence, if $W_{g}$ is computable, then so is
the halting problem. We arrive at a contradiction.
\end{proof}

\begin{remark} Theorem \ref{noncomputability-robustness} shows that
non-computability can be robust under the standard topological
structures in the study of natural phenomenons such as finding the
invariant sets of a dynamical system.
The robustness can occur in a very strong sense: the
non-computability nature of the basins of attraction persists, in a
continuous manner, for each and every function ``$C^1$ close to
$f$."
\end{remark}

\section{Semi-robust non-computability}

\label{Sec:RobustNoncomputability}

In this section, we continue our discussion with the map that
assigns to $(f, s)$ the basin of attraction of $s$ but for the
2-dimensional $C^1$ systems $dx/dt = f(x)$ defined on a compact disk
of $\mathbb{R}^2$. Unlike the 3-dimensional whole space case, we
show that this map is no longer robustly non-computable at any
$f$; but on the other hand, the map can be semi-robustly non-computable
at some $C^{\infty}$ functions.

We begin this section with several notional matters. Let $K$ denote
a closed disk of $\mathbb{R}^2$ centered at the origin with a
rational radius; in particular, let $\mathbb{D}$ denote the closed
unit disk of $\mathbb{R}^2$.  Let $\mathcal{V}(K)$ be the set of all
$C^1$ vector fields $K \to \mathbb{R}^2$ pointing inwards along the
boundary of $K$,
and let $\mathcal{O}_2$ be the set of all open subsets of
$\mathbb{R}^2$ equipped with the topology generated by the open
rational disks (disks with rational centers and rational radius) as
a subbase.

In order to show that the map assigns to $(f, s)$ the basin of
attraction of $s$ is not robustly non-computable at any $f\in
C^1(K)$, we need to identify a dense subset of $C^1(K)$ such that
the map is computable on this dense subset. This is where Peixoto's
well-known density theorem and the structurally stable systems come
into play. Peixoto's density theorem states that, on two-dimensional
compact manifolds, structurally stable systems form a dense open
subset in the set of all $C^1$ systems $\frac{dx}{dt} = f(x)$.
Hence, the map cannot be robustly non-computable at any $f\in
C^1(K)$ if we can show that the map is computable on the subset of
all structurally stable planar system. This is indeed the case as
shown in Theorem \ref{Th:ComputabilityBasins}.

We recall in passing that a planar dynamical system $dx/dt = f(x)$,
where $f\in C^1(K)$, is structurally stable if there exists some
$\varepsilon>0$ such that for all $g\in C^{1}(K)$ satisfying
$\left\Vert f-g\right\Vert _{1}\leq\varepsilon$, the trajectories of
$dy/dt = g(y)$ are homeomorphic to the trajectories of $dx/dt =
f(x)$, i.e. there exists some homeomorphism $h$ such that if
$\gamma$ is a trajectory of $dx/dt = f(x)$, then $h(\gamma)$ is a
trajectory of $dy/dt = g(y)$. Moreover, the homeomorphism $h$ is
required to preserve the orientation of trajectories by time.

For a structurally stable planar system $x^{\prime}=f(x)$ defined on
$K$, it has only finitely many equilibrium points and periodic
orbits, and all of them are hyperbolic. Recall (see Section
\ref{Subsec:DynamicalSystem}) that a point $x_{0} \in K$ is called
an equilibrium point of the system if $f(x)=0$, since
any trajectory starting at an equilibrium stays there for all $t\in\mathbb{R}
$. If all the eigenvalues of $Df(x_{0})$ have non-zero real parts,
then  $x_{0}$ is called a hyperbolic equilibrium.
If both eigenvalues of $Df(x_{0})$ have negative real parts, then it
can be shown that $x_{0}$ is a sink. A sink attracts nearby
trajectories. If both eigenvalues have positive real parts, then
$x_{0}$ is called a source. A source repels nearby trajectories. If
the real parts of the eigenvalues have opposite signs, then $x_{0}$
is called a saddle. A saddle attracts some points (those lying in
the so-called stable manifold, which is a one-dimensional manifold
for the planar systems); repels other points (those lying in the
so-called unstable manifold, which is also a one-dimensional
manifold for the planar systems, transversal to the stable
manifold); and all trajectories starting in a neighborhood of a
saddle point but not lying on the stable manifold will eventually
leave this neighborhood.  A periodic orbits (or limit cycle) is a
closed curve $\gamma$ with the property that there is some $T>0$
such that $\phi(f, x)(T) = x$ for any $x\in\gamma$. Hyperbolic
periodic orbits have properties similar to hyperbolic equilibria;
for a planar system there are only attracting or repelling
hyperbolic periodic orbits. See \cite[p. 225]{Per01} for more
details.

In Theorem \ref{Th:ComputabilityBasins} below, we consider the case
where $K=\mathbb{D}$ only for simplicity. The same argument applies
to any closed disk with rational radius.  Before stating and proving
Theorem \ref{Th:ComputabilityBasins}, we presents two lemmas first.
The proofs of the lemmas can be found in \cite{GZ21}. Let
$\mathcal{SS}_2 \subset \mathcal{V}(\mathbb{D})$ be the set of all
$C^1$ structurally stable planar vector field defined on
$\mathbb{D}$.

\begin{lemma} \label{Lem:NumberSinks} The map $\Psi_{N}: \mathcal{SS}_2
\to \mathbb{N}$, $f\mapsto \Psi_{N}(f)$, is computable, where
$\Psi_{N}(f)$ is the number of the sinks of $f$ in $\mathbb{D}$.
\end{lemma}

\begin{lemma} The map $\Psi_{S}: \mathcal{SS}_2 \times \mathbb{N} \to \mathbb{R}^2\cup \{ \emptyset\}$ is computable, where
\[ \left\{ \begin{array}{lll}
(f, i) \mapsto \emptyset & & \mbox{if $i=0$ or $i\geq \Psi_{N}(f)$}
\\
(f, i) \mapsto \mbox{ith sink of $f$} & & \mbox{if $1\leq i\leq
\Psi_{N}(f)$.}
\end{array} \right. \]
\end{lemma}

\begin{theorem} \label{Th:ComputabilityBasins} The map $\Psi:
\mathcal{SS}_2 \times \mathbb{N} \to \mathcal{O}$ is computable,
where
\[ \left\{ \begin{array}{lll}
   (f, i) \mapsto \emptyset & & \mbox{if $i=0$ or $i\geq
   \Psi_{N}(f)$} \\
   (f, i) \mapsto W_{s} & & \mbox{if $1\leq i\leq \Psi_{N}(f)$ and
   $s=\Psi_{S}(f, i)$.} \end{array} \right. \]
\end{theorem}

\begin{proof} Let us fix an $f\in \mathcal{SS}_2$. Assume that
$\Psi_{N}(f)\neq 0$ and $s$ is a sink of $f$. In \cite{Zho09} and \cite{GZ21},
it is shown that:
\begin{itemize}
\item[(1)] $W_{s}$ is a r.e.~open subset of $\mathbb{D}\subseteq\mathbb{R}^2$;
\item[(2)] there is an algorithm that on input $f$ and $k\in \mathbb{N}$, $k>0$, computes a finite
sequence of mutually disjoint closed squares or closed ring-shaped
strips (annulus) such that:
\begin{itemize}
\item[(a)] each square contains
exactly one equilibrium point with a marker indicating if it
contains a sink or a source;
\item[(b)] each annulus contains exactly one
periodic orbit with a marker indicating if it contains an attracting
or a repelling periodic orbit;
\item[(c)] each square (resp. annulus) containing a sink (resp. an attracting periodic orbit) is
time invariant for $t>0$;
\item[(d)] the union of this finite sequence contains all equilibrium
points and periodic orbits of $f$, and the Hausdorff distance
between this union and the set of all equilibrium points and
periodic orbits is less than $1/k$;
\item[(e)] for each annulus, $1\leq i\leq p(f)$, the minimal distance
between the inner boundary (denoted as $IB_i$) and the outer
boundary (denoted as $OB_i$), $m_i = \min\{ d(x, y) : \, x\in IB_i,
y\in OB_i\}$, is computable from $f$ and $m_i>0$.
\end{itemize}
\end{itemize}

We begin with the case that $f$ has no saddle point. Since $W_s$ is
r.e.~open, there exists computable sequences $\{ a_n\}$ and $\{
r_n\}$, $a_n\in \mathbb{Q}^2$ and $r_n\in \mathbb{Q}$, such that
$W_s=\cup_{n=1}^{\infty}B(a_n, r_n)$. Let $A$ be the union of all
squares and annuli in the finite sequence containing a sink or an
attracting periodic orbit except the square containing $s$, and let
$B$ be the union of all sources and repelling periodic orbits. Then
$B$ is a computable closed subset of $\mathbb{D}$ \cite{GZ21}.
Hence, $\mathbb{D}\setminus B$ is a computable open subset of
$\mathbb{D}$. Since $f$ has no saddle, $W_s\subset
\mathbb{D}\setminus B$. List the squares in $A$ as $S_1, \ldots,
S_{e(f)}$ and annuli as $C_1, \ldots, C_{p(f)}$. Denote the center
and the side-length of $S_j$ as $CS_j$ and $l_j$, respectively, for
each $1\leq j\leq e(f)$.

We first present an algorithm -- the classification algorithm --
that determines whether $x\in W_s$ or $x$ is in the union of basins
of attraction of the sinks and attracting periodic orbits contained
in $A$. The algorithm works as follows: for each $x\in
\mathbb{D}\setminus B$, simultaneously compute

\[ \left\{ \begin{array}{l}
   d(x, a_n), n=1, 2, \ldots \\ \\
   d(\phi_{t}(x), CS_j), 1\leq j\leq e(f), t=1, 2, \ldots \\ \\
   \mbox{$d(\phi_{t}(x), IB_i)$ and $d(\phi_{t}(x), OB_i)$}, 1\leq
   i \leq p(f), t=1, 2, \ldots \end{array} \right. \]
where $\phi_t(x)=\phi(f, x)(t)$ is the solution of the system $dz/dt
= f(z)$ with the initial condition $z(0)=x$ at time $t$. (Recall
that the solution (as a function of time $t$) of the initial-value
problem is uniformly computable from $f$ and $x$ \cite{GZB07}.) Halt
the computation whenever one of the following occurs: (i) $d(x,
a_n)<r_n$; (ii) $d(\phi_{t}(x), CS_j) < l_j/2$ for some $t=l\in
\mathbb{N}$ ($l>0$); or (iii) $d(\phi_{t}(x), IB_i) < m_i/2$ and
$d(\phi_{t}(x), OB_i) < m_i$ for $t=l\in \mathbb{N}$ ($l>0$). If the
computation halts, then either $x\in W_{s}$ provided that $d(x,
a_n)<r_n$ or else $\phi_{t}(x)\in S_j$ or $\phi_{t}(x)\in C_i$ for
some $t=l>0$. Since $S_j$ and $C_i$ are time invariant for $t>0$,
each $S_j$ contains exactly one sink for $1\leq j\leq e(f)$, and
each $C_i$ contains exactly one attracting periodic orbit for $1\leq
i\leq p(f)$, it follows that either $x$ is in the basin of
attraction of the sink contained in $S_j$ if (ii) occurs or $x$ is
in the basin of attraction of the attracting periodic orbit
contained in $C_i$ if (iii) occurs. We note that, for any $x\in
\mathbb{D}\setminus B$, exactly one of the halting status, (i),
(ii), or (iii), can occur following the definition of $W_s$ and the
fact that $S_j$ and $C_i$ are time invariant for $t>0$. Let $W_{A}$
be the set of all $x\in \mathbb{D}\setminus B$ such that the
computation halts with halting status (ii) or (iii) on input $x$.
Then it is clear that $W_{s}\cap W_{A}=\emptyset$.

We turn now to show that the computation will halt. Since there is
no saddle, every point of $\mathbb{D}$ that is not a source or on a
repelling periodic orbit will either be in $W_s$ or converges to a
sink/attracting periodic orbit contained in $A$ as $t\to \infty$
(this is ensured by the structural stability of the system and
Peixoto's characterization theorem; see, for example, \cite[Theorem
3 on page 321]{Per01}).
Thus either $x\in W_s$ or $x$ will eventually enter some $S_j$ (or
$C_i$) and stay there afterwards for some sufficiently large
positive time $t$. Hence the condition (i) or (ii) or (iii) will be
met for some $t>0$.

To prove that $W_s$ is computable it is suffices to show that the
closed subset $\mathbb{D}\setminus W_{s} = W_{A}\cup B$ is
r.e.~closed; or, equivalently, $W_{A}\cup B$ contains a computable
sequence that is dense in $W_{A}\cup B$. To see this, we first note
that $\mathbb{D}\setminus B$ has a computable sequence as a dense
subset. Indeed, since $\mathbb{D}\setminus B$ is computable open,
there exist computable sequences $\{ z_i\}$ and $\{\theta_i\}$,
$z_i\in \mathbb{Q}^2$ and $\theta_i\in \mathbb{Q}$, such that
$\mathbb{D}\setminus B = \cup_{i=1}^{\infty}B(z_i, \theta_i)$. Let
$\mathcal{G}_{l}=\{ (m/2^{l}, n/2^{l}): \, \mbox{$m, n$ are integers
and $-2^{l}\leq m, n\leq 2^{l}$}\}$ be the $\frac{1}{2^l}$-grid on
$\mathbb{D}$, $l\in \mathbb{N}$. The following procedure produces a
computable dense sequence of $\mathbb{D}\setminus B$: For each input
$l\in \mathbb{N}$, compute $d(x, z_i)$, where $x\in \mathcal{G}_l$
and $1\leq i\leq l$ and output those $\frac{1}{2^l}$-grid points $x$
if $d(x, z_i)<\theta_i$ for some $1\leq i\leq l$. By a standard
paring, the outputs of the computation form a computable dense
sequence, $\{ q_i\}$,  of $\mathbb{D}\setminus B$. Then the
algorithm will enlist those points in the sequence which fall inside
$W_{A}$, say $\tilde{q}_1, \tilde{q}_2, \ldots$. Clearly, $\{
\tilde{q}_j\}$ is a computable sequence. If we can show that $\{
\tilde{q}_j\}$ is also dense in $W_{A}$, then $W_{A}\cup B$ contains
a computable dense sequence. The conclusion comes from the fact that
$B$ is a computable closed subset; hence $B$ contains a computable
dense sequence.

It remains to show that $\{ \tilde{q}_j\}$ is dense in $W_{A}$. It
suffices to show that, for any $x\in W_{A}$ and any neighborhood
$B(x, \epsilon)\cap W_{A}$ of $x$ in $W_{A}$, there exists some
$\tilde{q}_{j_0}$ such that $\tilde{q}_{j_0}\in B(x, \epsilon)\cap
W_{A}$, where $\epsilon>0$ and the disk $B(x, \epsilon)\subset
\mathbb{D}\setminus B$. We begin by recalling a well-known fact that
the solution $\phi_{t}(x)$ of the initial value problem $dx/dt =
f(x)$, $\phi_{0}(x)=x$, is continuous in time $t$ and in initial
condition $x$. In particular, the following estimate holds true for
any time $t>0$ (see e.g.~\cite{BR89}):
\begin{equation} \label{error-initial-condition}
\|\phi_{t}(x) - \phi_{t}(y)\| \leq \| x - y\|e^{Lt}
\end{equation}
where $x=\phi_{0}(x)$ and $y=\phi_{0}(y)$ are initial conditions,
and $L$ is a Lipschitz constant satisfied by $f$. (Since $f$ is
$C^1$ on $\mathbb{D}$, it satisfies a Lipschitz condition and a
Lipschitz constant can be computed from $f$ and $Df$.) Since $x\in
W_{A}$, the halting status on $x$ is either (ii) or (iii). Without
loss of generality we assume that the halting status of $x$ is (ii).
A similar argument works for the case where the halting status of
$x$ is (iii).  It follows from the assumption that $d(\phi_{t}(x),
S_j)<l_j/2$ for some $1\leq j\leq e(f)$ and some $t=l>0$. Compute a
rational number $\alpha$ satisfying $0<\alpha < (l_j/2) -
d(\phi_{t}(x), S_j)$ and compute another rational number $\beta$
such that $0< \beta < \epsilon$ and $\| y_1 - y_2\|e^{l\cdot L}<
\alpha$ whenever $\| y_1 - y_2\|<\beta$. Then for any $y\in B(x,
\beta)$,
\begin{eqnarray*}
& & d(\phi_{t}(y), S_j) \\
& \leq & d(\phi_{t}(y), \phi_{t}(x)) + d(\phi_{t}(x), S_j)
\\
& \leq & \alpha + d(\phi_{t}(x), S_j) < (l_j/2) - d(\phi_{t}(x),
S_j) + d(\phi_{t}(x), S_j) = l_j/2
\end{eqnarray*}
which implies that $B(x, \beta)\subset W_{A}$. Since $B(x,
\beta)\subset B(x, \epsilon) \subset \mathbb{D}\setminus B$ and $\{
q_i\}$ is dense in $\mathbb{D}\setminus B$, there exists some
$q_{i_0}$ such that $q_{i_0}\in B(x, \beta)$. Since $B(0,
\beta)\subset W_{A}$, it follows that $q_{i_0}=\tilde{q}_{j_0}$ for
some $j_0$. This shows that $\tilde{q}_{j_0}\in B(x, \epsilon)\cap
W_A$.

We turn now to the general case where saddle point(s) is present. We
continue using the notations introduced for the special case where
the system has no saddle point. Assume that the system has the
saddle points $d_{m}$, $1\leq m\leq d(f)$ and $D_m$ is a closed
square containing $d_m$, $1\leq m\leq d(f)$. For any given $k\in
\mathbb{N}$ ($k>0$), the algorithm constructed in \cite{GZ21} will
output $S_j$, $C_i$, and $D_m$ such that each contains exactly one
equilibrium point or exactly one periodic orbit, the (rational)
closed squares and (rational) closed annuli are mutually disjoint,
each square has side-length less than $1/k$, and the Hausdorff
distance between $C_i$ and the periodic orbit contained inside $C_i$
is less than $1/k$, where $1\leq j\leq e(f)$, $1\leq m\leq d(f)$,
and $1\leq i\leq p(f)$. For each saddle point $d_m$, it is proved in
\cite{GZB12} that the stable manifold of $d_m$ is locally computable
from $f$ and $d_m$; that is, there is a Turing algorithm that
computes a bounded curve -- the flow is planar and so the stable
manifold is one dimensional -- passing through $d_m$ such that
$\lim_{t\to \infty}\phi_{t}(x_0)=d_m$ for every $x_0$ on the curve.
In particular, the algorithm produces a computable dense sequence on
the curve. Pick two points, $z_1$ and $z_2$, on the curve such that
$d_m$ lies on the segment of the curve from $z_1$ to $z_2$. Since
the system is structurally stable, there is no saddle connection;
i.e.~the stable manifold of a saddle point cannot intersect the
unstable manifold of the same saddle point or of another saddle
point. Thus, $\phi_{t}(z_1)$ and $\phi_{t}(z_2)$ will enter $C_B$
for all $t\leq -T$ for some $T>0$, where $C_{B}=(\cup \{ S_j : \,
s_j\in B\}) \cup (\cup \{ C_i: \, p_i\subset B\})$. We denote the
curve $\{ \phi_{t}(z_1): \, -T \leq t\leq 0\} \cup \{ z: \mbox{$z$
is on the stable manifold of $d_m$ between $z_1$ and $z_2$}\} \cup
\{ \phi_{t}(z_2): \, -T \leq t\leq 0\}$ as $\Gamma_{d_m}$. Let
$\widetilde{C}=C_B\cup \{ \Gamma_{d_m}: \, 1\leq m\leq d(f)\}$. Then
$\widetilde{C}$ is a computable compact subset in $\mathbb{D}$.
Moreover, every point in $\mathbb{D}\setminus \widetilde{C}$
converges to either a sink or an attracting periodic orbit because
there is no saddle connection. Using the classification algorithm
and a similar argument as above we can show that $W_{A}\cap
(\mathbb{D}\setminus \widetilde{C})$ is a computable open subset in
$\mathbb{D}\setminus \widetilde{C}$ and thus computable open in
$\mathbb{D}$. Since $W_{A}\subset \mathbb{D}\setminus B$ and
$W_{A}\cap \Gamma_{d_m}=\emptyset$, it follows that
\begin{eqnarray*}
& & d_{H}\left(\mathbb{D}\setminus (W_{A}\cap (\mathbb{D}\setminus
\widetilde{C})), \, \mathbb{D}\setminus (W_{A}\cap
(\mathbb{D}\setminus B))\right) \\
& = & d_{H}\left((\mathbb{D}\setminus W_A)\cup C_{B}, \,
(\mathbb{D}\setminus W_A)\cup B\right) \\
& \leq & d_{H}(C_{B}, B) < \frac{1}{k}.
\end{eqnarray*}
We have proved that there is an algorithm that, for each input $k\in
\mathbb{N}$ ($k>0$), computes an open subset $U_{k}=W_{A}\cap
(\mathbb{D}\setminus \widetilde{C})$ of $\mathbb{D}$ such that
$U_{k}\subset W_{A}$ and $d_{H}(\mathbb{D}\setminus U_{k}, \,
\mathbb{D}\setminus W_{A})<\frac{1}{k}$. This shows that $W_{A}$ is
a computable open subset of $\mathbb{D}$. (Recall an equivalent
definition for a computable open subset of $\mathbb{D}$: an open
subset $U$ of $\mathbb{D}$ is computable if there exists a sequence
of computable open subsets $U_k$ of $\mathbb{D}$ such that $U=\cup
U_{k}$ and $d_{H}(\mathbb{D}\setminus U_k, \, \mathbb{D}\setminus
U)\leq \frac{1}{k}$ for every $k\in \mathbb{N}\setminus \{ 0\}$.)
\end{proof}

\begin{corollary} For every $f\in \mathcal{SS}_2$ there is a
neighborhood of $f$ in $C^1(\mathbb{D})$ such that the function
$\Psi$ is (uniformly) computable in this neighborhood.
\end{corollary}

\begin{proof} The corollary follows from Peixoto's density theorem
and Theorem \ref{Th:ComputabilityBasins}.
\end{proof}

\begin{remark} In Theorem \ref{Th:ComputabilityBasins}, the closed unit disk
$\mathbb{D}$ can be replaced by closed disks of radius $k\in
\mathbb{N}$.
\end{remark}

Since $\mathcal{SS}_2$ is dense in $C^1(\mathbb{D})$, it follows
from Theorem \ref{Th:ComputabilityBasins} that the map $\Psi: (f,
s)\mapsto W_s$ cannot be robustly non-computable for planar $C^1$
flows in $\mathbb{D}$. On the other hand, it can be semi-robustly
non-computable, as shown in the following theorem.  Combined with
Theorem \ref{Th:ComputabilityBasins} the example below also
indicates that $\Psi$ is $\mathcal{SS}_2$ - removably
non-computable.

\begin{theorem} \label{Th:RobustNoncomputable} There exists a
computable $C^{\infty}$ function $h: \mathbb{R}^2 \to \mathbb{R}^2$,
$h\in \mathcal{V}(K)$, such that $h$ has a unique computable
equilibrium point $s$ - a sink - and the basin of attraction $W_s$
of $s$ is non-computable, where $K$ is the disk centered at the
origin with radius 3.  Moreover, for any $C$-neighborhood $U$ of $h$
in $\mathcal{V}(K)$, there exists a computable $C^{\infty}$ function
$\hat{h}$ such that $\hat{h}\in U$, $\hat{h}$ has a unique
computable equilibrium point $\hat{s}$ - a sink  - and the basin of
attraction $W_{\hat{s}}$ of $\hat{s}$ is non-computable.
\end{theorem}

\begin{proof} We make use of an example constructed in \cite{Zho09}:
    there is a  $C^{\infty}$ function $f: \mathbb{R}\to \mathbb{R}$
    uniformly computable in $\mathbb{R}$ such that the system
    \[ dx/dt = h(x) \]
    has a unique equilibrium point at $s=(0, 0)$, where $x=(x_1, x_2)$,
    $h(x_1, x_2) = (h_1(x_1, x_2), h_2(x_1, x_2))$, $h_1(x_1, x_2) =
    x_{1}f(x_1^2 + x_2^2) - x_2$ and $h_2(x_1, x_2) = x_{2}f(x_1^2 +
    x_2^2) + x_1$. This unique equilibrium point is a sink, and the
    basin of attraction $W_s$ of $s$ is the open disk $D_{\alpha}=\{
    (x_1, x_2)\in \mathbb{R}^2: \, x_1^2+x_2^2 < \alpha\}$, where
    $\alpha > 1/2$ is a left-computable but non-computable real number.
    Since $\alpha$ is non-computable, the basin of attraction $W_s$ is a
    non-computable open subset of $\mathbb{R}^2$. It is clear that $h$
    is a computable $C^{\infty}$ function. \\

    The key features of $f$ can be summarized as follows:
    \begin{itemize}
        \item[(1)]
        $f(w)<0$ for all $w\in (-1/4, 1/2)$ (ensuring $(0, 0)$ being a
        sink);
        \item[(2)] $f(w)=0$ for all $\alpha \leq w < 2$ and $f(w)<0$ for all $w\in [0,
        \infty)\setminus [\alpha, 2]$ (ensuring $W_s=D_{\alpha}$);
    \end{itemize}
    In polar coordinates, the equation $dx/dt = h(x)$ can be rewritten
    as $dr/dt = 2rf(r)$, $d\theta / dt = 1$, where $r=x_1^2 + x_2^2$.
    Since $f(w)<0$ for all $w>2$, it follows that $dr/dt < 0$ along the
    boundary of $K$. Hence the trajectories of the given planar system
    are transversal to and move inwards along the boundary of $K$;
    subsequently, $h\in \mathcal{V}(K)$.

    We turn now to prove the semi-robustness of non-computability at $h$.
    First we show that a small enough horizontal shift of $f$ is ``$C$
    close to $f$." Indeed,  since $f$ is uniformly computable in
    $[-1,16]$, there exists a computable function $\theta:
    \mathbb{N}\to \mathbb{N}$  (assuming that $\theta(n)> n+2$ without
    loss of generality) such that $| f(x) - f(y) | < 2^{-(n+2)}$
    whenever $| x - y | < 2^{-\theta(n)}$ for all $n\in \mathbb{N}$ and
    $x, y\in [-1,16]$. Let $g$ be a $2^{-\theta(n)}$ horizontal shift
    of $f$: $g(x) = f(x - 2^{-\theta(n)})$.
    Then $| f(x) - g(x) | = | f(x) - f(x - 2^{-\theta(n)}) | <
    2^{-(n+2)}$.

    Now for any $0<\epsilon < 1/4$, let $U$ be the
    $\epsilon$-neighborhood of $h$ (in $C$-norm) in $\mathcal{V}(K)$. We
    wish to find a computable function $g$ such that the system
    \begin{equation} \label{h-hat}
    dx/dt = \hat{h}(x_1, x_2)
    \end{equation}
    has a unique sink at $(0, 0)$ and the basin of attraction is
    non-computable, where $\hat{h}(x_1, x_2) = (\hat{h}_1(x_1, x_2),
    \hat{h}_2(x_1, x_2))$, $\hat{h}_1(x_1, x_2)=x_{1}g(x_1^2+x_2^2) -
    x_2$,  and $\hat{h}_2(x_1, x_2)=x_{2}g(x_1^2+x_2^2) + x_1$. Clearly,
    if $g$ is computable, so is $\hat{h}$.
    Let $g$ be a horizontal shift of $f$ with shifting distance less
    than 1 yet to be determined. Since on $K$ we have
    \begin{eqnarray*}
        & & \| h - \hat{h}\|_{K} \\
        & = & \max_{x\in K} \{ \| h_1(x) - \hat{h}_{1}(x)\|, \| h_2(x) -
        \hat{h}_{2}(x)\| \} \\
        & \leq & \max_{\sqrt{x_1^2+x_2^2}\leq 4} \{ |x_1|\, |f(x_1^2 +
        x_2^2) - g(x_1^2 + x_2^2)|, |x_2|\,
        |f(x_1^2 + x_2^2) - g(x_1^2 + x_2^2)| \} \\
        & = & \max_{x_1^2+x_2^2\leq 16} \{ |x_1|\, |f(x_1^2 + x_2^2) -
        g(x_1^2 + x_2^2)|, |x_2|\, |f(x_1^2 + x_2^2) - g(x_1^2 + x_2^2)| \}
        \\
        & \leq & 4 |f(x_1^2 + x_2^2) - g(x_1^2 + x_2^2)|
    \end{eqnarray*}
    it follows that if $|f(x_1^2 + x_2^2) - g(x_1^2 + x_2^2)|\leq
    \epsilon/4$, then $\| h - \hat{h}\|_{K} < \epsilon$. We can now
    determine the shift distance: pick $n$ such that $2^{-n} \leq
    \epsilon/4$ and take $g$ to be the $2^{-\theta(n)}$ horizontal shift
    of $f$: $g(x)=f(x - 2^{-\theta(n)})$.  Then $g$ is a computable
    function and $|f(x_1^2 + x_2^2) - g(x_1^2 + x_2^2)| = |f(x_1^2 +
    x_2^2) - f(x_1^2 + x_2^2 - 2^{-\theta(n)})|\leq \epsilon/4$, which
    in turn implies that $\hat{h}\in U$.

    In polar coordinates the system $dx/dt = \hat{h}(x)$ can be written
    as $dr/dt=2rg(r)$, $d\theta/dt = 1$, where $r=x_1^2+x_2^2$. The
    polar form indicates that the origin $(0, 0)$ is the unique
    equilibrium point of the system $dx/dt = \hat{h}(x)$. Moreover,
    since $g(w)<0$ for all $w > 2+ 2^{-\theta(n)}$, it follows from the
    polar form that the trajectories of the system (\ref{h-hat}) are
    transversal to and move inwards along the boundary of $K$. Hence
    $\hat{h}\in \mathcal{V}(K)$. Now since
    $f(w)<0$ in the interval $(-1/4, 1/2)$ and the shift distance is
    $2^{-\theta(n)}$ satisfying $2^{-\theta(n)} < 2^{-(n+2)} <
    \epsilon/8 < 1/32$, it follows that
    $0\in (-\frac{1}{4}+2^{-\theta(n)}, \frac{1}{2}+2^{-\theta(n)})$ and
    $g(w)<0$ in this interval, which ensures that the origin $(0, 0)$
    remains a sink of the new system. Since $g(w)=0$ for all $\alpha +
    2^{-\theta(n)} \leq w < 2 + 2^{-\theta(n)}$ and $g(w) < 0$ for all
    $w\in [0, 3]\setminus [\alpha + 2^{-\theta(n)}, 2+ 2^{-\theta(n)}]$,
    the basin of attraction of $(0, 0)$ is the disk $D=\{ (x_1, x_2)\in
    \mathbb{R}^2: \, x_1^2+x_2^2 < \alpha + 2^{-\theta(n)}\}$. Since
    $\alpha$ is non-computable and $2^{-\theta(n)}$ is a rational
    number, $\alpha + 2^{-\theta(n)}\texttt{}$ is a non-computable
    number, which further implies that $\sqrt{\alpha + 2^{-\theta(n)}}$
    is a non-computable number. Hence, the basin of attraction is
    non-computable.
\end{proof}

\bibliographystyle{alpha}
\bibliography{ContComp}

\end{document}